\documentclass[11pt]{article}
\usepackage[top=2.5cm, bottom=2.5cm, left=2.5cm, right=2.5cm] {geometry}

\usepackage{amsmath,amssymb,amsthm,color,comment,lipsum}

\newcommand{\R}{\ensuremath{\mathbb{R}}}

\newcommand{\eps}{\varepsilon}

\newcommand{\bP}{\mathbb{P}}
\newcommand{\cP}{\mathcal{P}}

\theoremstyle{plain}
\newtheorem{theorem}{Theorem}

\newtheorem{lemma}[theorem]{Lemma}

\theoremstyle{remark}
\newtheorem{remark}[theorem]{Remark}

\theoremstyle{definition}
\newtheorem{defn}[theorem]{Definition}
\newtheorem{example}[theorem]{Example}
\newtheorem{definition}[theorem]{Definition}

\newcounter{enumctr}

\newcommand{\D}{\mathcal{D}}
\newcommand{\F}{\mathbf{F}}

\newcommand{\RR}{\mathbb{R}}

\newcommand{\bel}[1]{\begin{equation}\label{#1}}

\newcommand{\be}{\begin{equation}}
\newcommand{\ba}{\begin{eqnarray}}
\newcommand{\ea}{\end{eqnarray}}

\newcommand{\qe}{\end{equation}}

\newcommand{\supp}{\rm{supp}}

\newcommand{\Om}{\Omega}

\begin{document}

\title{Ergodicity of scalar stochastic differential equations with H\"older continuous coefficients}
\author{Luu Hoang Duc
  \and
  Tat Dat Tran\\
  [4ex]\centering{{\bf Working paper}}
  \and
  J\"urgen Jost
  }

\newcommand{\Addresses}{{% additional braces for segregating \footnotesize
  \bigskip
  \footnotesize

  Luu Hoang Duc, \textsc{Max-Planck-Institut f\"ur Mathematik in den Naturwissenschaften,\\\hspace*{3,0cm} Institute of Mathematics, Viet Nam Academy of Science and Technology}\par\nopagebreak
  \textit{E-mail address}: \texttt{duc.luu@mis.mpg.de, lhduc@math.ac.vn}

  \medskip

  Tat Dat Tran, \textsc{Max-Planck-Institut f\"ur Mathematik in den Naturwissenschaften}\par\nopagebreak
  \textit{E-mail address}: \texttt{trandat@mis.mpg.de}

  \medskip

  J\"urgen Jost, \textsc{Max-Planck-Institut f\"ur Mathematik in den Naturwissenschaften\\
  \hspace*{2,3cm} Santa Fe Institute for the Sciences of Complexity, Santa Fe, NM 87501, USA}\par\nopagebreak
  \textit{E-mail address}: \texttt{jost@mis.mpg.de}
}}

\date{}
\maketitle

\tableofcontents

\begin{abstract}
It is well-known that for a one dimensional stochastic differential equation driven by Brownian noise, with coefficient functions satisfying the assumptions of the Yamada-Watanabe theorem \cite{yamada1,yamada2} and the Feller test for explosions \cite{feller51,feller54}, there exists a unique stationary distribution with respect to the Markov semigroup of transition probabilities. We consider systems on a restricted domain $D$ of the phase space $\R$ and study the rate of convergence to the stationary distribution. Using a geometrical approach that uses the so called {\it free energy function} on the density function space, we prove that the density functions, which are solutions of the Fokker-Planck equation, converge to the stationary density function exponentially under the Kullback-Leibler {divergence}, thus also in the total variation norm. The results show that there is a relation between the Bakry-Emery curvature dimension condition and the dissipativity condition of the transformed system under the Fisher-Lamperti transformation. Several applications are discussed, including the Cox-Ingersoll-Ross model and the Ait-Sahalia model in finance and the Wright-Fisher model in population genetics.
\end{abstract}

{\bf Keywords:}
stationary distributions, invariant measure, Fokker-Planck equation, Kullback-Leibler divergence, Cox-Ingersoll-Ross model, Ait-Sahalia model, Wright-Fisher model.

%%%%%%%%%%%%%%%%%%%%%%%%%%%%%%%%%%%%%%%%%%%%

\section{Introduction}

Stochastic differential equations with H\"older continuous coefficients arise as  models in many sciences.  An instance is the  Wright-Fisher diffusion model for the genetic drift of alleles with mutation \cite{ewens}; in the one dimensional case (that is, when there are only two alleles present), it  is of the form
\begin{equation*}
dX_t = [\theta_1 - (\theta_1 + \theta_2) X_t]dt + \sqrt{X_t(1-X_t)}dW_t,\quad  X_t \in [0,1], 
\end{equation*}
where we assume that the mutation rates $\theta_1,\theta_2 > 0$. Another example is the Cox-Ingersoll-Ross \cite{cox} model for the  short term interest rate
\begin{equation*}
dX_t = k(\theta -X_t)dt + \sigma \sqrt{X_t}dW_t,\quad  X_t \geq 0, 
\end{equation*}
where $k,\theta,\sigma >0$ satisfy $k\theta \geq \frac{\sigma^2}{2}$.\\
For such nonlinear systems, the diffusion coefficient is only  H\"older continuous on a restricted domain \cite{delbaen}, therefore the existence and uniqueness of a solution cannot be proved using classical arguments like contraction mappings, but one rather needs to invoke the Yoshida and Watanabe theorem \cite{yamada1, yamada2,ikeda}. Moreover, the solution in general does not depend differentiably on the initial values in the state space, and thus the linearization method for studying the stability problem fails to apply here. \\
In this paper, we study the Markov semigroup generated by such a system; this is a strong Feller process \cite{feller51, feller54} provided that the Feller test of explosion succeeds. In particular, assume that $D$ is a Polish space (complete metric and separable). A probability measure $\mu$ on $D$ is called {\it stationary (invariant)}\index{measure!invariant} with respect to the (Markov) diffusion process $X_t$ with  semigroup\index{strongly continuous semigroup} $(T_t)_{t\ge 0}$ on $C^\infty_0(D)$ and  generator $G$ if
$$\int_{D} T_t f(y) \mu(dy)=\int_{D} f(y) \mu(dy), \quad \forall t\ge 0, f\in C^\infty_0(D),$$
or equivalently (due to \cite{fukushima} Theorem 2.3)
$$\int_{D} G f(y) \mu(dy)= 0, \quad \forall f\in C^\infty_0(D).$$
It is called {\it reversible} if\index{measure!reversible}
$$\int_{D} g(y)T_t f(y) \mu(dy)=\int_{D} f(y) T_t g(y)\mu(dy), \quad \forall t\ge 0, f,g\in C^\infty_0(D),$$
or equivalently
$$\int_{D} g(y)G f(y) \mu(dy)=\int_{D} f(y)G g(y) \mu(dy) , \quad \forall f,g\in C^\infty_0(D).$$
Also, it is well-known that the density function of the transition probability satisfies the Fokker-Planck equation. The existence of stationary measures with respect to the Markov semigroup is deduced from the Krylov-Bogoliubov theorem \cite{prato}. However, under the assumptions on reversibility and boundary conditions, we can prove directly that there exists a unique stationary measure for the Markov semigroup which can be written as a Gibbs measure $\mu_\infty(dy) = u_\infty(y)dy = \frac{e^{-\psi(y)}}{Z}dy$. \\
Another important issue is the rate of convergence to the stationary distribution. It is well known from the Harris theorem \cite{harris} that if there exists a Lyapunov type function, an initial distribution under the Markov semigroup will converge exponentially to the stationary distribution. Such a Lyapunov function is constructed in Hairer \cite{hairer} for the system with additive noise, however it is not explicitly given in general. \\
Our approach here is to present an explicit construction of a Lyapunov type function for the Markov semigroup on the space of probability measures. Specifically, we consider only those systems in which there exists a so-called {\it free energy functional} defined as the sum of the {\it relative entropy} and the {\it potential energy} (see, e.g. \cite{jordan}) along the flow of densities $\{u(t,x,\cdot)\}_{t\ge 0}$. We then prove that the free energy functional decreases in time $t$ along the flow of densities.\\
Our main result is the convergence rate of the densities under the flow to the stationary density using the geometrical method by Bakry-Emery \cite{bakry, gross}. By considering the backward generator of the Markov semigroup $(L^*,\D(L^*))$, the {\it carr\'e du champ operator}, the {\it iterated carr\'e du champ operator} of $L^*$ and the {\it Bakry-Emery curvature-dimension condition}, we derive a good estimate for the exponential rate of convergence of the density function $u(t,x,\cdot)$ to $u_\infty(\cdot)$ with respect to the Kullback-Leibler divergence.\\
It is interesting to note that for one dimensional systems, there is a relation between the Bakry-Emery curvature dimension condition and the dissipativity condition of the transformed system under the so called {\it Fisher-Lamperti transformation} which is supposed to be continuous. This transformation helps to prove the existence of a global random attractor (see \cite{arnold}, \cite{crauel08, crauel99}, \cite{schmalfuss}) for the Cox-Ingersoll-Ross model \eqref{cox}, the Wright-Fisher model \eqref{wright} and  the Ait-Sahalia model \eqref{sahalia}.

%%%%%%%%%%%%%%%%%%%%%%%%%%%%%%%%%%%%%%%%%%%%

\section{Existence and uniqueness of the stationary distribution}

We consider the following one dimensional stochastic system
\begin{equation}\label{SDE}
dX_t = a(X_t) dt + b(X_t) dW_t
\end{equation}
where $D\in \Big\{[d_-, d_+], (-\infty, d_+], [d_-,\infty), \R\Big\}$, and $a, b: D \to \R$ are $C^3$-functions on the interior of a domain $D \subset \R$ ($D$ can contain $\infty$), which satisfy the following assumptions.
\paragraph{Hypothesis A:} The functions $a,b$ satisfy
\begin{itemize}
\item there exists a positive increasing concave function $L: [0,\infty) \to [0,\infty)$ such that $|a(x)-a(y)| \leq L(|x-y|)$ for all $x,y \in D$ and $\int_{0^+}^p L(r)^{-1}dr = \infty$ for every $p >0$; 
\item there exists a positive increasing function $\rho: [0,\infty) \to [0,\infty)$ such that $|b(x)-b(y)| \leq\rho(|x-y|)$ for all $x,y \in D$, and $\int_{0^+}^p \rho(r)^{-2}dr = \infty$ for every $p>0$.
\item boundary conditions: $a(d_-) >0 > a(d_+); b(d_-) =b(d_+) =0$, $b(x)>0$ for $x\in (d_-,d_+)$. 
\end{itemize}   

%%%%%%%%%%%%%%%%%%%%%%%%%%%%%%

\begin{theorem}\label{existence-uniqueness}
Under Hypothesis A for the coefficient functions $a,b$ there exists a unique solution of \eqref{SDE} with initial condition $X_0 = x \in D$.
\end{theorem}
\begin{proof}
First, we extend (if necessary) the coefficient functions of \eqref{SDE} from the domain $D$ into the whole real line by considering the system
\begin{equation}\label{SDE1}
dX_t = \bar{a}(X_t)dt + \bar{b}(X_t)dW_t
\end{equation}
where $\bar{a}(x) \equiv a(x), \bar{b}(x) \equiv b(x)$ for all $x\in D$ and $\bar{a}(x) = a(d_-) (>0), \bar{b}(x) = b(d_-) (=0)$ for all $x < d_-$, $\bar{a}(x) = a(d_+) (<0), \bar{b}(x) = b(d_+) (=0)$ for all $x>d_+$. It is easy to check that $\bar{a}, \bar{b}$ then satisfy Hypothesis A. Hence by applying the Yamada and Watanabe theorem (see \cite{yamada1, yamada2} or the Ikeda and Watanabe theorem \cite[Theorem 3.2, p.~168]{ikeda}), there exists a unique solution $\bar{X}_t$ of \eqref{SDE1} with initial condition $X_0$.

Next, we need to prove that all solutions of \eqref{SDE1} stay in the domain $D$ almost surely, i.e. $\mathbb{P}(\bar{X}_t(\cdot,x) \in D, \forall t \in [0,T]) = 1$ for any $x\in D$. It is enough to prove that in case the left boundary of $D$ is $d_-$ then  $\mathbb{P}(\bar{X}_t(\cdot,x) \geq d_-, \forall t \in [0,T]) = 1$.  To do that, we follow the Deelstra and Delbaen technique \cite{delbaen}  by introducing for any $\epsilon >0$ stopping times
\begin{eqnarray*}
\tau_1 &=& \inf \{t \in [0,T]| \bar{X}_t < d_- - \epsilon \} \\
\sigma_1^1 &=& \inf \{t \in [0,T] | \bar{X}_t < d_- - 2 \epsilon\}\\
\sigma_1^2 &=& \inf \{ t>\tau_1 | \bar{X}_t =d_-\}\\
\sigma_1 &=& \sigma_1^1 \wedge \sigma_1^2 \wedge T.
\end{eqnarray*} 
Trivially $\tau_1 < \sigma_1 \leq T$. Define the set $A := \{\inf_{u \in [0,T]} \bar{X}_u < d_--2\epsilon\}$. If $\sigma_1^1 \leq \sigma_1^2 $ then $\sigma_1 = \sigma_1^1$, hence $\bar{X}_{\sigma_1} - \bar{X}_{\tau_1} = -\epsilon$ due to the continuity of $\bar{X}$ over time. But on the other hand,
$\bar{b}(u)=b(d_-) =0$ and $\bar{a}(u) = a(d_-) >0$ for every $u \in [\tau_1,\sigma_1]$, thus by using Ito's formula for the stopping times,
\[
\bar{X}_{\tau_1} -\bar{X}_{\sigma_1} = \int_{\tau_1}^{\sigma_1} \bar{a}(u)du +\int_{\tau_1}^{\sigma_1} \bar{b}(u)dW_u > 0, 
\]
which is a contradiction. We conclude that on this set, $\sigma_1^2 < \sigma_1^1\leq T$ almost surely. By definition, we conclude that $\bar{X}_{\sigma_1} = d_-$ on $A$.\\
Similarly, define some more stopping times:
\begin{eqnarray*}
\tau_2 &=& \inf \{\sigma_1<t \in [0,T]| \bar{X}_t < d_- - \epsilon \} \wedge T\\
\sigma_2^1 &=& \inf \{\sigma_1<t \in [0,T] | \bar{X}_t < d_- - 2 \epsilon\}\\
\sigma_2^2 &=& \inf \{ t>\tau_2 | \bar{X}_t =d_-\}\\
\sigma_2 &=& \sigma_2^1 \wedge \sigma_2^2 \wedge T.
\end{eqnarray*} 
Analogously on $A$, $\tau_1<\sigma_1<\tau_2<\sigma_2 \leq T $ and $\bar{X}_{\tau_2}=d_-$. Therefore, we can repeat this argument and conclude that on $A$ there exists a strictly increasing sequence of stopping times $\tau_1<\sigma_1 < \ldots < \tau_n < \sigma_n < \ldots \leq T$. This sequence converges to a limit $\tau \leq T$, hence all the subsequences $(\tau_n)$ and $(\sigma_n)$ have to converge to $\tau$. However, $\bar{X}_{\tau_n} = d_--\epsilon$ while $\bar{X}_{\sigma_n} = d_-$ on $A$. By the continuity of $\bar{X}$ over time, it follows that $A$ has measure zero, or equivalently $\mathbb{P} (\inf_{u \in [0,T]} \bar{X}(u)< d_--2\epsilon ) =0$. Since it is true for any $\epsilon >0$, we have proved that $\mathbb{P} (\inf_{u \in [0,T]} \bar{X}(u)< d_-) =0$.  

Finally, by the definition of $\bar{a}, \bar{b}$, the solution $\bar{X}$ is actually also a solution of \eqref{SDE}, thus we conclude the existence and uniqueness of \eqref{SDE} on the domain $D$ for the time interval $[0,T]$. As $T$ varies, we actually can extend the solution of \eqref{SDE1} to the whole  time interval $[0,\infty)$.

\end{proof}

%%%%%%%%%%%%%%%%%%%%%%%%%%%%%%%%%%%%%%%%%%%%

It is well-known that the density function of the Markov semigroup with transition probability $\cP_t(x,O):= \bP\{X_t(\cdot,x) \in O| X_0 = x\}$ satisfies the Fokker-Planck equation
\begin{eqnarray}\label{fokkerplanck}
\partial_t u(t,x,y) &=& L u(t,x,y)\nonumber\\
&=& -\partial_y [a(y)u(t,x,y)] + \frac{1}{2} \partial_{yy} [b^2(y) u(t,x,y)] \nonumber\\
&=& \partial_y\Big( -a(y)u(t,x,y) + \frac{1}{2}\partial_y \big(b^2(y) u(t,x,y)\big) \Big) .  
\end{eqnarray}
In this section, we only consider  systems that satisfy special conditions which ensure the existence of a free energy function. To do that, we first consider the definition.

%%%%%%%%%%%%%%%%%%%%%%%%%%%%%%

\begin{defn}
The  family of densities $\{u(t,\cdot)\}_{t\ge 0}$ on a  $\sigma-$finite measure space $(D,\mu)$ is said to satisfy the condition $I(A,\psi)$ if it solves a diffusion equation of the form
\begin{eqnarray}\label{DEE}
\partial_t u(t,y) &=  & \partial_y \Big(A(y) u(t,y) \partial_{y} \big[\log u(t,y) + \psi(y)\big]\Big),
\end{eqnarray}
where $A(y) = \frac{1}{2}b^2(y)$. This family of densities $\{u(t,\cdot)\}_{t\ge 0}$  is said to satisfy the condition $II(A,\psi)$ if in addition to  $I(A,\psi)$,  one has
\begin{equation}\label{cond1}
\int_{D} e^{-\psi(y)}\mu(dy) < \infty.
\end{equation}
\end{defn}

%%%%%%%%%%%%%%%%%%%%%%%%%%%%%%

We rewrite a system that satisfies condition $I(A,\psi)$ in the one dimensional case as
\[
\partial_t u(t,x,y) = \partial_y \Big(A(y)u(t,x,y) \partial_y \big(\log u(t,x,y) +\psi(y)\big)\Big). 
\]
We introduce  the potential energy function $\psi(y)$ 
\begin{equation}
\psi(y) := \log A(y) - G(y)=\log \frac{1}{2} b(y)^2 - G(y)
\end{equation}
where $G(y)$ is the indefinite integral of $\frac{a(y)}{A(y)}$, i.e. $G(y) = \int^y \frac{a(y)}{A(y)} dy$. It is easy to check that 
\[
A(y)u(t,x,y) \partial_y \big(\log u(t,x,y) +\psi(y)\big)=-a(y)u(t,x,y) + \frac{1}{2}\partial_y \big(b^2(y) u(t,x,y)\big).
\]
Thus equation \eqref{fokkerplanck} satisfies the $I(A,\psi)$ condition. We will then write the Kolmogorov (forward/backward) operators in this form, i.e.
\[
Lu(y)= \partial_y \Big(A(y)\partial_y u(y)\Big)  +\partial_y \Big(A(y)u(y) \Big) \partial_y\psi(y),
\]
and
\[
L^*u(x)= \partial_x \Big(A(x)\partial_x u(x)\Big)  - \partial_x\psi(x)A(x)\partial_x u(x).
\]
We need the following assumption.
\paragraph{Hypothesis B:} The potential energy function $\psi$ satisfies the condition $II(A,\psi)$ and
\begin{equation}\label{cond1.1}
A(y)e^{-\psi(y)} \Big|_{\partial D} = \lim \limits_{y\to \partial D} e^{G(y)} =0 
\end{equation}
If $\psi$ satisfies the Hypothesis $B$, then we have immediately.

%%%%%%%%%%%%%%%%%%%%%%%%%%%%%%

\begin{lemma}\label{symm} The measure
\[
\mu_{\infty}(dy)=u_{\infty}(y)dy=\frac{e^{-\psi(y)}}{Z}dy,
\]
(where $Z=\int_{D}e^{-\psi(y)}dy <\infty$ due to condition (\ref{cond1}), is the normalization coefficient), 
is reversible with respect to $L^{*}$, i.e.
\begin{equation}\label{reversible}
\int_{D} f L^{*}g d\mu_{\infty}  = \int_{D} g L^{*}f d\mu_{\infty},\quad \forall f,g \in C^2(D). 
\end{equation}
\end{lemma}

%%%%%%%%%%%%%%%%%%%%%%%%%%%%%%

\begin{proof}
Integrating by parts, we have
\begin{eqnarray*}
\eqref{reversible}&\Leftrightarrow& \int_D f [(Ag_y)_y - A\psi_y g_y] e^{-\psi}dy = \int_D g [(Af_y)_y - A\psi_y f_y] e^{-\psi}dy\\
 &\Leftrightarrow& \int_D \Big[f (Ag_y)_y - g(Af_y)_y \Big]e^{-\psi}dy = \int_D A (fg_y-gf_y)\psi_y e^{-\psi}dy \\
 &\Leftrightarrow& \int_D \Big[A(fg_{yy} - gf_{yy})  + A_y(fg_y-gf_y)\Big] e^{-\psi}dy = \int_D A(fg_y-gf_y)d(-e^{-\psi})\\
 &\Leftrightarrow& \int_D \Big[A(fg_{yy} - gf_{yy})  + A_y(fg_y-gf_y)\Big] e^{-\psi}dy = -(fg_y-gf_y)Ae^{-\psi} \Big|_{\partial D}
  +  \int_D d\Big[A(fg_y-gf_y)\Big]e^{-\psi}.
\end{eqnarray*}
From assumption \eqref{cond1.1}, the first term on the right hand side of the last equation is zero, which then implies that
\begin{eqnarray*}
\eqref{reversible}&\Leftrightarrow& \int_D \Big[A(fg_{yy} - gf_{yy})  + A_y(fg_y-gf_y)\Big] e^{-\psi}dy = \int_D \Big[A(fg_y-gf_y)_y + A_y(fg_y-gf_y)\Big]e^{-\psi}dy\\
&\Leftrightarrow& \int_D A(fg_{yy} - gf_{yy})   e^{-\psi}dy =  \int_D A(fg_y-gf_y)_y e^{-\psi}dy.
\end{eqnarray*}
The last equation holds for every $ f,g \in C^2(D)$. This implies the proof. 
\end{proof}

%%%%%%%%%%%%%%%%%%%%%%%%%%%%%%

\begin{theorem}\label{Gibbsmeasure}
Assume that $a,b$ satisfy Hypotheses A, B. Then there exists a unique stationary measure for the Markov semigroup $\cP_t(x,O):= \bP\{X_t(\cdot,x) \in O| X_0 = x\}$. Moreover, the stationary distribution is of the form
\begin{equation}
\mu_\infty(dy) = u_\infty(y)dy = \frac{e^{-\psi(y)}}{Z}dy
\end{equation} 
\end{theorem}

%%%%%%%%%%%%%%%%%%%%%%%%%%%%%%

\begin{proof}
Since the original system is one dimensional, we can prove by using Lemma \ref{symm} that the invariant measure, if it exists, is reversible. We can then follow similar arguments as in \cite{dat} to prove the existence and uniqueness of the invariant measure.
\end{proof}

%%%%%%%%%%%%%%%%%%%%%%%%%%%%%%

\begin{remark}
In general, when the probability measure $\mu_\infty$ is not reversible, we need to prove that the Markov semigroup $\cP_t(x,\cdot)$ is strongly Feller and tight. The existence is then a consequence of the Krylov-Bogoliubov theorem. Moreover, we need to prove that $\cP_t$ is irreducible and thus regular by applying the Khasminskii theorem, and then apply the Doob theorem to conclude that $\lim \limits_{t\to \infty} \cP_t(x,O) = \mu_\infty(O)$ which in particular proves the uniqueness of $\mu_\infty$. \\
For more information about this method, see Prato and Zabczyk \cite{prato}.
\end{remark}

% % % % % % % % % % % % % % % %

We now would like to find conditions so that the process will never reach the boundary points with probability one. To do so, we first recall some concepts from potential theory:
\begin{defn}
Let there be given a reversible process $(X_t)_{t\ge 0}$ in $D$ with generator $L^*$ and reversible measure $u_{\infty}(x)dx$. The corresponding Dirichlet form is 
\begin{equation}
\label{2.5}
D(f) := - \int_D \! f(x) L f(x) u_{\infty}(x)dx.
\end{equation}
For $c>0$, we also define the quadratic form by
\begin{equation}
D_c (f) = D(f) + c\int_D\!  f(x)^2u_{\infty}(x)dx.
\end{equation}
\end{defn}
\begin{defn}
We define the capacity of an open set $O \subseteq D$ by
\begin{equation}
\label{2.6}
cap_c(O) := \inf_{f \in \F^c_{O}} D_c(f)\,,
\quad \quad
\F^c_ O := \big\{ f\in \D_c \, :  \; f(x) \ge 1, \:\forall  x\in O \big\}
\end{equation}
For an arbitrary set $B \subseteq D$ the capacity of $B$ is defined as
\begin{equation}
\label{2.7}
cap_c(B) := \inf_{O \textrm{ open} \: :  \; O \supseteq B}
cap_c(O)
\end{equation}
\end{defn}
\begin{theorem}\label{inside}
If in addition to the hypothesis A, the coefficients satisfy 
\bel{internal}
c(x)=-a'(x)+b''(x)b(x)+b'^2(x)\equiv c>0, \forall x\in {D},
\qe
then the solution is in $D^o$ with probability one.
\end{theorem}
\begin{proof}
 First, observe that an eigenfunction $f$ of $L^*$ for the eigenvalue $c$ is given by 
$$
f_\eps(x) = \frac{u_{\infty}(d_-+\eps)}{u_{\infty}(x)} \frac{G(x)}{G(d_-+\eps)},
$$
where $G(x) = \int_x^{d_+} \frac{u_{\infty}(y)}{b^2(y)}dy$. Similar to the arguments  in \cite{bertini08}, we can prove by direct computation that, under this additional condition \eqref{internal}, 
\[
cap_{c}(\partial D) = \lim \limits_{\eps \to 0} D_c(f_\eps) = 0.
\] 
Then, we apply a classical result in  potential theory (see for example \cite{fu1}, Theorem 4.3.1, page 103) to show that the points with capacity zero are exceptional, which means that the boundary can never be reached with probability one.
\end{proof}

\section{Rate of convergence to the stationary distribution}

In this section we present a direct construction of a Lyapunov--type function for the Markov semigroup on the space of probability measures. It is worth  noting  here that the existence of a Lyapunov function for a Markov semigroup on the phase space is assumed in Harris' theorem (see \cite{harris} or Hairer \cite{hairer}), but such a function in general is not explicitly given.  

%%%%%%%%%%%%%%%%%%%%%%%%%%%%%% 

\begin{defn}
For a family of densities $\{u(t,\cdot)\}_{t\ge 0}$ on a  $\sigma-$finite measure space $(D,\mu)$ with condition $I(A,\psi)$,
\begin{itemize}
\item the potential energy functional is defined by 
\begin{equation}
\Psi(u(t,\cdot)) := \int\limits_{D} u(t,y) \psi(y) \mu (dy);   
\end{equation}
\item the (negative) {entropy} functional is defined by
\begin{equation}
	\begin{split}
		S_{\mu}(f) = \int_{D} f \log f d\mu;
	\end{split}
\end{equation}
\item the {free energy} functional is defined by 
\begin{eqnarray}
F(u(t,\cdot)) := \int\limits_{D} u(t,y) \Big(\log u(t,y) +\psi(y)\Big) \mu(dy) 
= S_{\mu}(u(t,\cdot))+\Psi(u(t,\cdot)).   
\end{eqnarray}
\end{itemize}
\end{defn}

%%%%%%%%%%%%%%%%%%%%%%%%%%%%%%

\begin{defn}
	Let $f_1, f_2$ be  densities on a  $\sigma-$finite measure space $(D,\mu)$. The relative entropy\index{relative entropy} (Kullback--Leibler divergence\index{Kullback--Leibler divergence}) of $f_1$ with respect to $f_2$ is 
	
	\begin{equation*}
	D_{\mathrm{KL}}(f_1 \| f_2) :=
	\begin{cases}
	\int_{D} f_1(y) \log \frac{f_1(y)}{f_2(y)} \mu(dy), & \text{ if } \supp(f_1) \subset \supp(f_2)\\
	\infty, \text{ otherwise} 
	\end{cases}
	\end{equation*}
\end{defn}

%%%%%%%%%%%%%%%%%%%%%%%%%%%%%%

Before presenting our main theorem, we need the following assumption.
\paragraph{Hypothesis C:} There exists a $\rho >0$
such that
\begin{equation}\label{cond2}
\frac{1}{2}b b_{yy} - a_y + \frac{ab_y}{b}\geq \rho >0.
\end{equation}

%%%%%%%%%%%%%%%%%%%%%%%%%%%%%%

\begin{theorem}\label{mainthm}
	Assume that $a,b$ satisfies Hypotheses A, B, C. Then the density function $u(t,x,\cdot)$ converges to $u_\infty(\cdot)$ exponentially w.r.t.  the Kullback-Leibler divergence, i.e.
	\begin{equation}\label{KL}
	D_{\mathrm{KL}}(u(t,x,\cdot)\|u_\infty(\cdot)) \leq e^{-\rho t} D_{\mathrm{KL}}(u(0,x,\cdot)\|u_\infty(\cdot)) 
	\end{equation}
\end{theorem}

% % % % % % % % % % % % % % % %

\begin{proof}
	The proof follows the approach in Tran et al. \cite{dat} and is divided into several steps.\\
	{\bf Step 1:} First, we can prove that  $F(u(t,x,\cdot))$ decreases in time $t$ along the flow of densities. In fact,
	\begin{equation}
	\partial_t F(u(t,x,\cdot)) = -\int_D I(t,x,y)dy,
	\end{equation}
	where
	\[
	I(t,x,y) = \frac{1}{2}b^2(y) u(t,x,y) \Big[\partial_y \psi(y) + \frac{1}{u(t,x,y)} \partial_y u(t,x,y)\Big]^2 \geq 0.  
	\]
	{\bf Step 2:} Next, assume that there exists  a unique stationary distribution $\mu_{\infty}(dy) = u_{\infty}(y)dy$. We focus on the rate of the convergence of $u$ to $u_{\infty}$. Putting
	\[
	h(t,x,y):=\frac{u(t,x,y)}{u_{\infty}(y)},
	\]
	we shall investigate the rate of the convergence of $h$ to 1.\\
	The stationary density is the Gibbs density function
	\[
	u_{\infty}(y) = \frac{e^{-\psi(y)}}{Z},
	\]
 which is independent of $x$. Thus
	\[
	\log u_{\infty}(y) + \psi(y) =-\log Z.
	\]
	Since $Z$ is independent of $y$, this implies 
	\[
	\partial_y (\log u(t,x,y) + \psi(y)) = \partial_y \Big(\log \frac{u(t,x,y)}{u_{\infty}(y)}\Big) + \partial_y (\log u_{\infty}(y) + \psi(y)) =  \partial_y (\log h(t,x,y)).
	\]
	Therefore, using  similar computations as in \cite{dat}, we can derive a partial {differential} equation for $h$ from that of $u$ as follows.
	\begin{equation}
	\partial_t h(t,x,y) = \partial_y\Big[A(y)\partial_y h(t,x,y)\Big] - \partial_y\psi(y) A(y)\partial_y h (t,x,y)= L^{*}h(t,x,y),
	\end{equation}
	where 
	\begin{equation}\label{A}
	A(y) := \frac{1}{2}b(y)^2.
	\end{equation}
	Also, the difference of the current  and the final free energy is equal to the relative entropy (Kullback-Leibler divergence) between the corresponding densities and also equal to the (negative) entropy of their ratio with respect to the stationary probability measure:
	\begin{equation}
	F(u(t,x,\cdot))-F_{\infty}(u_\infty(\cdot)) = D_{\mathrm{KL}}(u(t,x,\cdot)\|u_{\infty}(\cdot))=S_{\mu_{\infty}}(h(t,x,\cdot)) \ge 0.
	\end{equation}
	Moreover, the rate of change of the free energy functional is equal to the negative of the entropy production 
	\begin{equation}
	\frac{d}{dt}S_{\mu_{\infty}}(h(t,x,\cdot))=\partial_t F(u(t,x,\cdot)) = -J_{\mu_{\infty}}(h) := - \int_{D} \frac{A(y)[\partial_yh(t,x,y)]^2}{h(t,x,y)} \mu_{\infty}(dy).
	\end{equation} 
	 {\bf Step 3:} Now we check the convergence rate of the density function to the stationary distribution. For that we need to apply the geometric theory of Markov operators. From now on, for simplicity, we will write 
	 \[
	 q_y:=\partial_y q(y),\  q_{yy}:=\partial_{yy}q(y),\  q_{yyy}:=\partial_{yyy}q(y).
	 \]
	 Following \cite{gross}, we consider an operator $(L^*,D(L^*))$ defined on a measure space $(D, \mu)$ of the form
	 \begin{eqnarray*}
	 	L^*f&:=&\partial_y\Big[A\partial_yf \Big] - \partial_y \psi A\partial_yf\\
	 	&=& \Big[Af_y \Big]_y - \psi_yAf_y\\
	 	&=& Af_{yy} +  A_y f_y- \psi_yAf_y, \forall f\in \mathcal{A} = L^2(\Om, \mu)\cap D(L^*).
	 \end{eqnarray*}
	 where 
	 \begin{equation}\label{psi_y}
	 \psi_y = \frac{A_y}{A} - \frac{a}{A}.
	 \end{equation}
	 Then, we define the carr\'e du champ operator of $L^*$ by 
	 \bel{eq:G1}
	 \Gamma(f,g)= \frac{1}{2}\Big(L^*(fg)-fL^*g-gL^*f\Big),\quad \forall f,g\in \mathcal{A}
	 \qe
	 and the iterated carr\'e du champ operator of $L^*$ by
	 \bel{eq:G2}
	 \Gamma_2(f,g)= \frac{1}{2}\Big(L^*\Gamma(f,g)-\Gamma(f,L^*g)-\Gamma(g,L^*f)\Big),\quad \forall f,g\in \mathcal{A}.
	 \qe
	 We will also denote $\Gamma(f,f)=\Gamma_1(f)$ and $\Gamma_2(f,f)=\Gamma_2(f)$ for short. We note that for the above operator $L^*$, we always have
	 \[
	 \Gamma(f,g)=A f_y g_y.
	 \]
	 A direct computation then shows that
	 \begin{eqnarray}
	 \Gamma_1(h) &=& \frac{1}{2} \Big[L^*(h^2) - 2 hL^*(h) \Big] = Ah_y^2  \label{gamma1}\\
	 \Gamma_2(h) &=& \frac{1}{2} \Big[L^*(\Gamma_1(h)) - 2 \Gamma(h,L^*(h)) \Big] \nonumber \\
	 &=& \Big( \frac{A_y^2}{4A} - \frac{A_{yy}}{2}  + A\psi_{yy} + \frac{\psi_yA_y}{2}\Big) \Gamma_1(h) + (Ah_{yy} + \frac{1}{2} A_yh_y)^2 \label{gamma2}.
	 \end{eqnarray}
	 Using \eqref{A} and \eqref{psi_y} and {\bf Hypothesis C}, we can estimate directly the first coefficient in \eqref{gamma2}
	 \[
	 \frac{A_y^2}{4A} - \frac{A_{yy}}{2}  + A\psi_{yy} + \frac{\psi_yA_y}{2} = \frac{1}{2}b b_{yy} - a_y + \frac{ab_y}{b} \geq \rho >0.
	 \]
	 Hence,  $L$ satisfies the {curvature-dimension condition} $CD(\rho,n)$ (see \cite{gross}), where $\rho>0$ and $n=\infty$, i.e. for all $f\in \mathcal{A}$ 
	 \bel{eq:CD}
	 \Gamma_2(f)\ge \rho \Gamma_1(f) +\frac{1}{n}(L^*f)^2, \quad \mu-a.e. 
	 \qe
	 Using  similar arguments as in \cite{dat}, since $L^*$ satisfies $CD(\rho,\infty)$, it follows from \cite{gross} that $\mu_{\infty}$ satisfies the logarithmic Sobolev inequality $LSI(\rho)$, i.e. for all densities $f$ we have 
	 \[
	 \int_{D} f \log f d\mu \le \frac{1}{\rho} \int_{D} \frac{1}{2f} |\nabla f|^2 d\mu.
	 \]
	  In particular, 
	  \[
	  S_{\mu_{\infty}}(h(t,x,\cdot)) \leq e^{-\rho t}  S_{\mu_{\infty}}(h(0,x,\cdot)).
	  \]
\end{proof}
\begin{remark}
Using the Csisz\'ar-Kullback-Pinsker inequality (see e.g. \cite{csiszar})
\bel{CKP}
\|\mu-\nu\|_{TV} \le \sqrt{\frac{1}{2} D_{\mathrm{KL}}(\mu\|\nu)}
\qe
where
\[
\|\mu-\nu\|_{TV}=\sup \{ |\mu(O) - \nu(O)| : O\text{ is an event to which probabilities are assigned} \}
\]
is the total variation metric (or statistical distance) between two probability measure $\mu$ and $\nu$, we combine \eqref{CKP} and 
\eqref{KL} to conclude that $P(t,x,dy) = u(t,x,y)dy$ converges to $\mu_\infty(dy)= u_\infty(y)dy$ with exponential rate $\rho$ in the total variation metric acting on the space of probability measures.\\
On the other hand, if we define the semigroup $S_t \mu(A) := \int_D P_t(A,y)\mu(dy)$ on the space of measures, then
\begin{eqnarray*}
\|S_t \mu - \mu_\infty\|_{TV} &=& \sup_{A} \Big|\int_D P(t,y,A)\mu(dy) - \mu_\infty(A) \Big|\\
&\leq& \sup_A \int_D |P(t,y,A) - \mu_\infty(A)| \mu(dy) \leq \|P(t,y,\cdot) -\mu_\infty\|_{TV}.
\end{eqnarray*} 
Therefore, $S_t\mu \rightarrow \mu_\infty$ in total variation norm as $t \to \infty$.
\end{remark}

%%%%%%%%%%%%%%%%%%%%%%%%%%%%%%%%%%%%%%%%%%%%

\section{Existence of the global random attractor}

In this section, we consider the situation in which the Fisher-Lamperti transformation can be applied, meaning that there exists a continuous transformation which transforms \eqref{SDE} into a system with additive noise (see e.g. \cite{neuenkirch}). We find that under the hypotheses A, B, C, there actually exists a global random attractor for the random dynamical system generated by \eqref{SDE}. This result is somewhat stronger than the existence of a stationary measure because it implies the convergence in the pathwise sense rather than the convergence in distribution only. \\ 
To do that, we first recall the basic notions of random dynamical systems.  Let $(\Omega,\mathcal{F},\mathbb{P})$ be a probability space. On this probability space we consider a measurable flow $\theta$
\[
\theta:\mathbb{R}\times\Omega\to\Omega
\]
for which $\theta_t(\cdot): \Omega \to \Omega$ is $\mathbb{P}$-preserving, i.e.\ $\mathbb{P}(\theta_t^{-1}(A)) = \mathbb{P}(A)$ for every $A \in \mathcal{F}$, $t \in \R$, and $(\theta_t)_{t\in \mathbb{R}}$ satisfies the group property, i.e.\ $\theta_{t+s} = \theta_t \circ \theta_s$ for all $t,s \in \mathbb{R}$. A general model for noise is the quadruple $(\Omega,\mathcal{F},\mathbb{P},(\theta_t)_{t\in \mathbb{R}})$ which is called a {\em metric dynamical system}. For system \eqref{SDE}, $\theta$ can be constructed as the Wiener shift w.r.t. the corresponding probability space of the Wiener process \cite{arnold}.

\begin{definition}
Let $X$ be a Banach space. Then we define a random dynamical system (RDS) as a  measurable mapping
\[
\varphi:\R^+\times \Omega\times X\to X
\]
satisfying the cocycle property 
\begin{eqnarray*}
\varphi(t+s,\omega,x)&=&\varphi(t,\theta_s\omega,\cdot)\circ\varphi(s,\omega,x)\qquad \text{for all }
t,\,s\in\mathbb{R}_+,\,\omega\in\Omega,\,x\in X,\qquad \\
\varphi(0,\omega)&=&{\rm id}_X.
\end{eqnarray*}
An RDS $\varphi$ is called {\em continuous} if each mapping $(t,x) \mapsto \varphi(t,\omega,x)$ is continuous.

\end{definition}

\begin{lemma}\label{Gronwall-new}
Given $\eps\in (0,1]$. Let $f: \RR^+ \to \RR^+$ be continuous. Assume that there exists a continuous nondecreasing super-linear function $k: \R^+ \to \R^+ $ with $k(t) =0$ iff $t=0$ such that 
$$
f(t)\le \eps + \int_0^t k(f(s))ds, \quad \forall t\ge 0.
$$
Then 
$$
f(t)\le \eps G^{-1}(G(1)+t), \quad \forall t\ge 0,
$$
where 
$$
G(t)=\int_0^t \frac{1}{k(s)}ds.
$$
\end{lemma}

\begin{proof}
Put $u(t)=\frac{f(t)}{\eps}$. Then it follows from the assumption that 
\begin{eqnarray*}
u(t)&\le& 1+\int_0^t \frac{k(f(s))}{\eps}ds\\
&\le& 1+\int_0^t k(u(s))ds \quad \text{(due to the super-linearity of $k$)}.
\end{eqnarray*}
Put $v(t)=1+\int_0^t k(u(s))ds$. It implies that $u(t)\le v(t)$ and $v(t)$ is nondecreasing. Moreover, because $k$ is nondecreasing we have $v'(t)=k(u(t))\le k(v(t))$ and $k(v(t))\ge k(v(0))=k(1)>0$. Therefore, $G(v(t))-G(v(0))=\int_0^t \frac{dv(s)}{k(v(s))}\le t$. Because $G'(t)=\frac{1}{k(t)}>0$, both $G$ and $G^{-1}$ are increasing. This implies the proof.
\end{proof}

\begin{theorem}\label{RDS}
Assume there exists the inverse map $Q^{-1}(x)$ of the indefinite integral $Q(x)$ of $\frac{1}{b(x)}$, i.e. $Q(x) = \int^x \frac{1}{b(y)}dy$ such that both $Q(x)$ and $Q^{-1}(x)$ are continuous on $D$ and $Q(D)$. Assume there exists a continuous nondecreasing super-linear function $k: \R^+ \to \R^+ $ with $k(0) =0$ such that
$c(y) := \frac{a(Q^{-1}(y))}{b(Q^{-1}(y))} + \frac{1}{2} b_y(Q^{-1}(y)) $ satisfies
\[
|c(x)-c(y)| \leq k(|x-y|),\ \forall x,y \in D.
\]  
Then the solution of \eqref{SDE} generates a continuous random dynamical system $\varphi: \R^+ \times \Omega \times D \to D$ defined by $\varphi(t,\omega)X_0 = X_t(\omega,X_0)$.

\end{theorem}

\begin{proof}
Introduce the transformation $Y_t= Q(X_t)$. Then by  Ito's formula it follows from \eqref{SDE} that
\begin{equation}\label{transformed}
dY_t = c(Y_t) dt + dW_t.
\end{equation}
It is easy to see that $Y_t(\omega, Y_0)$ is continuous w.r.t. $t \in \R^+$. On the other hand, by the comparison principle in the one dimensional space $D$, we have $Y_t(\omega,Y_1) \geq Y_t(\omega,Y_0) \geq 0$ for all $t\geq 0$ a.s. if $Y_1 > Y_0\ge 0$ a.s. Therefore
\begin{eqnarray*}
Y_t(\omega,Y_1) - Y_t(\omega,Y_0) &=& Y_1-Y_0 + \int_0^t \big[c(Y_s(\omega,Y_1)) - c(Y_s(\omega,Y_0))\big]ds\\
&\leq&  |Y_1-Y_0| + \int_0^t k(|Y_s(\omega,Y_1) - Y_s(\omega,Y_0)|)ds.
\end{eqnarray*}
Then by Lemma~\ref{Gronwall-new} above and noting that $G^{-1}(G(1)+t)$ is bounded uniformly in $t\in [0,T]$, it follows that $Y_t(\omega,Y_0)$ is a continuous function w.r.t. $Y_0$ uniformly in $t\in[0,T]$. By using the triangle inequality
\[
|Y_{t_1}(\omega,Y_1) - Y_t(\omega,Y_0)| \leq |Y_{t_1}(\omega,Y_1) - Y_t(\omega,Y_1)| + |Y_{t}(\omega,Y_1) - Y_t(\omega,Y_0)|,
\]
it follows that $Y_t(\omega,Y_0)$ is continuous w.r.t. $(t,Y_0)$. Therefore $X_t(\omega,X_0) = Q{^{-1}}(Y_t(\omega,Y_0))$ is continuous w.r.t. $(t,X_0)$.\\
By applying Arnold \cite[Chapter 2]{arnold}, we get from the existence and uniqueness of a solution of \eqref{SDE} that the system generates a random dynamical system $\varphi(t,\omega)X_0:= X_t(\omega,X_0)$ which is continuous w.r.t. $(t,X_0)$. Moreover, if $\psi(t,\omega)$ is the RDS generated by the transformed system \eqref{transformed} then $\psi(t,\omega) = Q \circ \varphi(t,\omega) \circ Q^{-1}$.
\end{proof}

\begin{remark}
In general we can extend the random dynamical system $\varphi$ to the whole two sided time interval $\R$ instead of $\R^+$ by defining $\varphi(t,\omega) := \varphi(-t,\theta_t \omega)^{-1}$ for any $t\leq 0$, see Arnold \cite[Chapter 1]{arnold} for more details.
\end{remark}

\begin{definition}
A global random attractor of an RDS $\varphi$  is a random compact set $A$ which is invariant, i.e. $\varphi(t,\omega)A(\omega) = A(\theta_t \omega)$ for all $t\geq 0$, and which attracts all the bounded sets in the pullback sense, i.e.,
\[
\lim \limits_{t\to \infty} d(\varphi(t,\theta_{-t}\omega)M | A(\omega)) = 0,
\]
for all bounded sets $M$, where $d(M|A):= \sup_{x\in M}d(x,A)$ is the semi-Hausdorff distance. 
\end{definition}

\begin{remark}
(i), The random attractor acts as the support of the invariant measure of the generated random dynamical systems. In fact, 
\begin{equation}\label{invariantmeasure}
\lim \limits_{t \to \infty} \varphi(t,\theta_{-t}\omega) \mu_{\infty} = \delta_{a(\omega)},
\end{equation}
where $\delta_a$ is a Dirac measure and the limit in \eqref{invariantmeasure} is in the sense of weak topology. We recommend Crauel \cite{crauel08} and Arnold \cite[Chapter 1, Proposition 1.8.4, pp.~42]{arnold} for a presentation of the relation between invariant measures of the random dynamical system and the stationary distribution of the corresponding Markov semigroup. \\
\bigskip
(ii), Under the assumptions of theorems \ref{mainthm} and \ref{RDS}, a direct computation shows that 
\begin{eqnarray*}
	\frac{dc(y)}{dy}  &=& \frac{d}{dy}\Big( \frac{\tilde{a}(y)- \frac{1}{2} \tilde{b}_y(y)}{\tilde{b}(y)}\Big) 
	= \frac{d\tilde{c}(x)}{dx}\frac{dx}{dy}\\
	&=& \frac{d}{dx}\Big(\frac{a(x)}{b(x)}-\frac{1}{2} \frac{db(x)}{dx}\Big)b(x)\\
	&=& -\Big(\frac{1}{2}b(x) \frac{d^2b(x)}{dx^2} - \frac{da(x)}{dx} + \frac{a(x) \frac{db(x)}{dx}}{b(x)}\Big) \le -\rho.
\end{eqnarray*}
Hence,
$$
[c(y_1)-c(y_2)](y_1-y_2) \le \max_{z\in D} c'(z) (y_1-y_2)^2 \le -\rho (y_1-y_2)^2,
$$ 
which proves the dissipativity of the coefficient function $c$. Therefore, by applying similar arguments as in \cite{kloeden09}, we can prove that there exists a global random attractor $\bar{A}$ for the RDS $\psi$ generated by the transformed system \eqref{transformed}, i.e. for any $x_0\in D$
\[
\lim \limits_{t\to \infty} Q \circ \varphi(t,\theta_{-t} \omega) \circ Q^{-1}(x_0) = \bar{A}(\omega). 
\]
However, since $Q$ and $Q^{-1}$ are not assumed to have bounded derivatives, in general it does not follow that $A(\omega):= Q^{-1}(\bar{A}(\omega))$ is the global random attractor of $\varphi$, i.e. we do not have
\[
\lim \limits_{t\to \infty} \varphi(t,\theta_{-t} \omega) \circ Q^{-1}(x_0) = Q^{-1}(\bar{A}(\omega)).
\] 
Below we will present several examples in which there does exist a random attractor for the generated RDS. \\
\bigskip
(iii) We note  that in general for higher dimensional systems, the relation between the dissipative constant of the transformed system \eqref{transformed} and the Bakry-Emery (BE) curvature dimension is still an open question.
\end{remark}

%%%%%%%%%%%%%%%%%%%%%%%%%%%%%%%%%%%%%%%%%%%%

\begin{example}[Cox-Ingersoll-Ross interest rate model]
We consider the Cox-Ingersoll-Ross \cite{cox} (CIR) model for the short term interest rate
\begin{equation}\label{cox}
dX_t = k(\theta -X_t)dt + \sigma \sqrt{X_t}dW_t,  X_t \geq 0, 
\end{equation}
where $k,\theta,\sigma >0$ such that $k\theta \geq \frac{\sigma^2}{2}$. Feller \cite{feller51} proved that the process is nonnegative. The Markov semigroup is proved to be Feller by Duffie et al. \cite{duffie}. Given $X_0$, it is well known that $\frac{4k}{\sigma^2 (1-e^{-kt})} X_t$ follows a noncentral $\chi^2$ distribution with degree of freedom $\frac{4k\theta}{\sigma^2}$ and non-centrality parameter $\frac{4k}{\sigma^2 (1-e^{-kt})} X_0 e^{-kt}$. Moreover, $\lim \limits_{t\to \infty} \frac{1}{t} \int_0^t X_s = X_\infty$ by the law of large numbers, where $X_0$ follows a Gamma distribution with shape parameter $\frac{2k\theta}{\sigma^2}$ and scale parameter $\frac{\sigma^2}{2k}$ (see e.g. Deelstra and Delbaen \cite{deelstra1,deelstra2, deelstra3}).\\
In this case $b(y) = \sigma \sqrt{y}$ which is $\frac{1}{2}$-H\"older continuous in the domain $D:= [0,\infty)$, therefore the existence and uniqueness of a solution of \eqref{cox} can be derived from \cite{delbaen}. Note that the condition (\ref{internal}) in the CIR case is $c(x)\equiv k>0$ so $X_t >0$ for all $t\geq 0$ with probability one. \\
A direct computation to check conditions \eqref{cond1}, \eqref{cond1.1} and \eqref{cond2} gives us
\begin{eqnarray*}
\frac{1}{2}b b_{yy} - a_y + \frac{ab_y}{b} &=& \frac{k}{2} + \frac{1}{2x}(k\theta - \frac{\sigma^2}{4}) \geq \frac{k}{2}\geq\frac{k}{2} \\
\int_0^\infty e^{-\psi(y)}dy &=& \int_0^\infty y^{(\frac{2k\theta}{\sigma^2}-1)} e^{-\frac{2k}{\sigma^2}y}dy< \infty.
\end{eqnarray*}
Hence there exists a unique stationary measure for the corresponding Markov semigroup of system \eqref{cox}, with  rate of convergence $\frac{k}{2}$. \\
We can also prove that there exists a global random attractor. Indeed, the transformation $Y_t = Q(X_t) = \frac{2\sqrt{X_t}}{\sigma}$ gives us
\begin{equation}\label{coxtransform}
dY_t = c(Y_t)dt + dW_t= \Big[\frac{4k\theta-\sigma^2}{2\sigma^2} \frac{1}{Y_t} - \frac{k}{2} Y_t\Big]dt + dW_t.
\end{equation}
By the comparison principle in one dimensional space $D$, we have $Y_t(\omega,Y_1) \geq Y_t(\omega,Y_0) \geq 0$ for all $t\geq 0$ a.s. if $Y_1 \geq Y_0$ a.s. Therefore
\[
Y_t(\omega,Y_1) - Y_t(\omega,Y_0)\leq Y_1-Y_0-\frac{k}{2}\int_0^t \Big[Y_s(\omega,Y_1) - Y_s(\omega,Y_0)\Big]ds, 
\]
thus 
\begin{equation}\label{forwconv}
\Big|Y_t(\omega,Y_1) - Y_t(\omega,Y_0)\Big| \leq |Y_1-Y_0|e^{-\frac{k}{2}t}.
\end{equation}
Also, it is easy to check that the drift coefficient $c$ in \eqref{coxtransform} satisfies the dissipativity condition, i.e.
\[
(y_1 - y_2)(c(y_1) - c(y_2)) = \Big[ - \frac{4k\theta-\sigma^2}{2\sigma^2}  \frac{1}{y_1 y_2} - \frac{k}{2}\Big](y_1-y_2)^2 \leq -\frac{k}{2}(y_1-y_2)^2.
\]
Therefore, as seen in Kloeden et al. \cite{kloeden09}, there exists a pullback random attractor (and also a forward random attractor because of \eqref{forwconv}) $Y^*(\omega)$ which is invariant under the random dynamical system generated by \eqref{coxtransform}. We are going to prove that $X^*(\omega) = \frac{\sigma^2}{4}Y^*(\omega)^2$ is also a random attractor of the random dynamical system generated by \eqref{cox}.\\ 
Indeed, taking the expectation in both sides of \eqref{cox} gives
\[
dEX_t = k (\theta - EX_t) dt,
\]
thus 
\[
EX_t = \frac{\sigma^2}{4} EY_t^2 = \theta + [EX_0 - \theta] e^{-\frac{k}{2}t}.
\]
By the Borel-Cantelli lemma, it is easy to prove that for all $Y_0 \geq 0$
\[
\lim \limits_{t \to \infty} e^{-\frac{k}{2}t} Y_t(\omega,Y_0) = 0 \ \text{a.s.}.
\]
Hence
\begin{eqnarray*}
|X_t - X^*(\theta_t \omega)| &=& \frac{\sigma^2}{4} \Big|Y_t - Y^*(\theta_t \omega)\Big| \Big|(Y_t - Y^*(\theta_t \omega)) + 2 Y^*(\theta_t \omega)  \Big|\\
&\leq& \frac{\sigma^2}{4} |Y_0 - Y^*(\omega)|^2 e^{-kt} + \frac{\sigma^2}{4} |Y_0 - Y^*(\omega)| e^{-\frac{k}{2}t} Y^*(\theta_t \omega),
\end{eqnarray*}
which tends to zero as $t\to \infty$, proving the attraction of the random attractor $X^*(\omega)$.\\ 

\end{example}

%%%%%%%%%%%%%%%%%%%%%%%%%%%%%%

\begin{example}[Wright-Fisher model]
We consider the one dimensional  Wright-Fisher model \cite{dat}
\begin{equation}\label{wright}
dX_t = [\theta_1 - (\theta_1 + \theta_2) X_t]dt + \sqrt{X_t(1-X_t)}dW_t,  X_t \in [0,1], 
\end{equation}
where $\theta_1,\theta_2 \geq \frac{1}{2}$. For the existence and uniqueness of a solution of \eqref{wright}, see \cite{dat}, in this case $b(y)= \sigma \sqrt{y(1-y)}$ is $\frac{1}{2}$-H\"older continuous in the domain $D= [0,1]$. The same method to prove the existence and uniqueness of a stationary distribution is mentioned and proved in \cite{dat}. For this system, the condition \eqref{internal} becomes $c(x)\equiv \theta_1+\theta_2-1>0$, which is sufficient to conclude that the solution never reaches the boundary with probability one.\\
A direct computation to check conditions \eqref{cond1}, \eqref{cond1.1} and \eqref{cond2} gives us
\begin{eqnarray*}
\frac{1}{2}b b_{yy} - a_y + \frac{ab_y}{b} &=& \frac{(\theta_1-\frac{1}{4})(y-1)^2 + (\theta_2-\frac{1}{4})y^2}{y(1-y)} \geq 2\sqrt{(\theta_1-\frac{1}{4})(\theta_2-\frac{1}{4})}>0 \\
\int_0^1 e^{-\psi(y)}dy &=& \int_0^1 y^{2\theta_1-1} (1-y)^{2\theta_2-1} dy< \infty.
\end{eqnarray*}
Hence there exists a unique stationary measure for the corresponding Markov semigroup of system \eqref{wright}, with  rate of convergence  $\frac{k}{2}$. \\
To prove the existence of a global random attractor, observe that the transformation $Y_t = \arcsin 2 \sqrt{X_t (1-X_t)}$ gives us
\begin{equation}\label{RFtransform}
dY_t = \frac{\alpha_1 - (\alpha_1 + \alpha_2) X_t}{\sqrt{X_t (1-X_t)}}dt  + dW_t = \frac{\alpha_1 - (\alpha_1 + \alpha_2) \sin^2 \frac{Y_t}{2}}{\frac{1}{2}\sin Y_t}dt  + dW_t = c(Y_t)dt + dW_t,
\end{equation}
where $\alpha_i = \theta_i - \frac{1}{4}$. By the comparison principle, it follows that for any $Y_1\leq Y_2$ then $Y_t(\omega,Y_1) \leq Y_t(\omega,Y_2)$ a.s. Since $\cot $ is a decreasing function on $[0,\frac{\pi}{2}]$, it is then easy to check that
\begin{eqnarray*}
[c(y_2)- c(y_1)](y_2-y_1) &=& \Big[\alpha_1 \cot \frac{y_2}{2} - \alpha_2 \tan \frac{y_2}{2} - \alpha_1 \cot \frac{y_1}{2} + \alpha_2 \tan \frac{y_1}{2}\Big] (y_2-y_1)\\
&\leq&  - \Big[\tan \frac{y_2}{2} -\tan \frac{y_1}{2}\Big] (y_2-y_1) \\
&\leq& - (y_2-y_1)^2, 
\end{eqnarray*}
i.e. the drift coefficient $c$ in \eqref{RFtransform} satisfies the dissipativity condition. Again, by \cite{kloeden09}, it follows that there exists a random attractor (both pullback and forward) $Y^*(\omega)$ for the random dynamical system generated by \eqref{RFtransform}. Moreover,
\begin{eqnarray*}
d(Y_t(\omega,Y_2) - Y_t(\omega,Y_1)) &=& \Big[c(Y_t(\omega,Y_2)) - c(Y_t(\omega,Y_1))\Big]dt\\
&\leq& -\Big [\tan \frac{Y_t(\omega,Y_2)}{2} -\tan \frac{Y_t(\omega,Y_1)}{2}\Big] dt \\
&\leq& -\frac{1}{2}\Big[ Y_t(\omega,Y_2) - Y_t(\omega,Y_1)\Big]dt,
\end{eqnarray*}
which results in
\[
|Y_t(\omega,Y_2) - Y_t(\omega,Y_1)| \leq |Y_2-Y_1| e^{-\frac{1}{2}t}.
\]
Then similar to system \eqref{cox} there also exists a random attractor $X^*(\omega) = \sin^2 \frac{Y^*(\omega)}{2}$ of the random dynamical system generated by \eqref{wright}, due to the fact that
\begin{eqnarray*}
|X_t(\omega,X_2) - X_t(\omega,X_1)| &=& \Big| \frac{\cos Y_t(\omega,Y_2) -\cos Y_t(\omega,Y_1)}{2} \Big|  \\
&=& \big|\sin \frac{[Y_t(\omega,Y_2) - Y_t(\omega,Y_1)]}{2}\Big| \Big|\sin \frac{[Y_t(\omega,Y_2) + Y_t(\omega,Y_1)]}{2}\Big| \\
&\leq& \Big|\frac{[Y_t(\omega,Y_2) - Y_t(\omega,Y_1)]}{2}\Big|. 
\end{eqnarray*}

\end{example}

%%%%%%%%%%%%%%%%%%%%%%%%%%%%%%

\begin{example}[Ait-Sahalia interest rate model]
We now consider the Ait-Sahalia-type model for the stochastic interest rate (see Ait-Sahalia \cite{sahalia} and Szpruch et al. \cite{szpruch})
\begin{equation}\label{sahalia}
dx(t) = \Big[k(\theta-x(t)) - \alpha x(t)^r + \frac{\beta}{x(t)}\Big]dt + \sigma x(t)^p dW(t),
\end{equation}
where $k,\theta, \alpha, \beta, \sigma, r, p >0$ such that $r>1, p\geq \frac{1}{2}$ and $r+1\geq 2p$. System \eqref{sahalia} does not satisfy hypothesis A since it contains a singular value $0$ in the drift term. However, the coefficient functions of \eqref{sahalia} are locally Lipschitz continuous in $(0,\infty)$, thus by using the stopping time technique, Szpruch et al. \cite{szpruch} prove that there exists a unique solution of system \eqref{sahalia} which is positive for all $t>0$ given that $x_0 >0$. \\
Our direct computations for checking conditions \eqref{cond1} and \eqref{cond1.1} show that
\begin{eqnarray*}  
\int_0^\infty e^{-\psi(y)}dy &=& \int_0^\infty \frac{2}{\sigma^2 y^{2p}} e^{\int^y \frac{2k\theta}{\sigma^2}x^{-2p}dx - \frac{k}{(1-p)\sigma^2}y^{2(1-p)} - \frac{2\alpha}{\sigma^2(r+1-2p)}y^{r+1-2p} -\frac{\beta}{p\sigma^2}y^{-2p}}dy <\infty
\end{eqnarray*}
due to the fact that the dominant coefficients $ -\frac{\beta}{p\sigma^2}$ and $- \frac{2\alpha}{\sigma^2(r+1-2p)}$ are negative. Therefore there exists a unique stationary measure $\mu_\infty$ for the corresponding Markov semigroup.  \\
To check condition \eqref{cond2}, observe that
\begin{eqnarray*}  
\frac{1}{2}b b_{yy} - a_y + \frac{ab_y}{b} &=& k (1-p) + \frac{\beta(1+p)}{y^2} + \frac{k\theta p}{y} + \frac{\sigma^2p(p-1)}{2}y^{2p-2}+\alpha (r-p)y^{r-1}\\
&>&  k (1-p) + \frac{k\theta p}{y} + \frac{\sigma^2p(p-1)}{2}y^{2p-2}.
\end{eqnarray*}
We consider three cases.
\begin{itemize}
\item Case 1: $p \in [\frac{1}{2},1)$, then 
\[
\frac{1}{2}b b_{yy} - a_y + \frac{ab_y}{b} > k (1-p) + \frac{k\theta (1-p)}{y} + \frac{\sigma^2p(p-1)}{2}y^{2p-2} = (1-p)g(y),
\]
where $g: \R_+ \to \R$ is a convex function which has a minimum at $y_1 = \Big[\frac{k\theta}{p(1-p)\sigma^2}\Big]^{\frac{1}{2p-1}}$. Therefore, the sufficient condition that ensures the exponential rate of convergence to the stationary measure is
\[
(1-p)g(y_1) >0 \Leftrightarrow y_1 > \frac{\theta(2p-1)}{2(1-p)}
\]
\item Case 2: $p >1$, then 
\[
\frac{1}{2}b b_{yy} - a_y + \frac{ab_y}{b} > k (1-p) + \frac{k\theta (p-1)}{y} + \frac{\sigma^2p(p-1)}{2}y^{2p-2} = (p-1)h(y),
\]
where $h: \R_+ \to \R$ is a convex function which has a minimum at $y_2 = \Big[\frac{k\theta}{p(p-1)\sigma^2}\Big]^{\frac{1}{2p-1}}$. Therefore, the sufficient condition for an exponential rate of convergence to the stationary measure is
\[
(p-1)h(y_2) >0 \Leftrightarrow y_2 < \frac{\theta(2p-1)}{2(p-1)}.
\]
\item Case 3: $p=1$, then 
\[
\frac{1}{2}b b_{yy} - a_y + \frac{ab_y}{b} > \frac{k\theta}{y} + \alpha (r-1)y^{r-1} = q(y),
\]
where $q: \R_+ \to R$ is convex which has a minimum at $y_3 = \Big[\frac{k\theta}{\alpha(r-1)^2}\Big]^{\frac{1}{r}}$. In this case $q(y_3) >0$ thus we also have the exponential rate of convergence to $\mu_\infty$.
\end{itemize}
In all cases, the Fisher-Lamperti transformation can be written explicitly as $Q(x) = x^{1-p}$ for $p \ne 1$ and $Q(x) = \log x$ for $p=1$, therefore we can also prove that there exists a random attractor for the random dynamical system generated by \eqref{sahalia}. The proof is similar as above and will be omitted here.
\end{example}

%%%%%%%%%%%%%%%%%%%%%%%%%%%%%%%%%%%%%%%%%%%%

\section*{Acknowledgments}
This work was carried out at the Max Planck Institute for Mathematics in the Sciences (MPI MIS Leipzig). It was also partially supported by the Vietnam National Foundation for Science and Technology Development (NAFOSTED) under grant number 101.03-2014.42. 

%%%%%%%%%%%%%%%%%%%%%%%%%%%%%%%%%%%%%%%%%%%%

\Addresses


\begin{thebibliography}{1}
%
\bibitem{sahalia}
Y. Ait-Sahalia,
\newblock{Testing continuous time models of the spot interest rate.}
\newblock{\em Review of Financial Studies,} Vol.{\bf 9}, No.{\bf 2}, (1996), 385--426.
%
\bibitem{arnold}
L. Arnold.
\newblock{\em Random Dynamical Systems.}
\newblock{Springer, Berlin Heidelberg New York}, 1998.
%
\bibitem{kloeden09}
M. Garrido-Atienza, P. Kloeden, A. Neuenkirch.
\newblock{Discretization of stationary solution of stochastic systems driven by fractional brownian motion.}
\newblock{\em Appl. Math. Optim.} {\bf 69}, (2009), 151--172.
%
\bibitem{bakry}
D. Bakry, I. Gentil, M. Ledoux.
\newblock{\em Analysis and geometry of Markov diffusion operators. Vol. 38 of  Grundlehren der Mathematischen Wissenschaften [Fundamental Principles of Mathematical Sciences].}
\newblock{Springer, Cham}, 2014.
%
\bibitem{bertini08}
L. Bertini, L. Passalacqua.
\newblock{Modelling interest rates by correlated multi-factor CIR-like processes.}
\newblock{\em Available at SSRN 1175702} {\bf }, (2008).
%
\bibitem{cox}
J.C. Cox, J.E. Ingersoll, S.A. Ross.
\newblock{A theory of the term structure of interest rates.}
\newblock{\em Econometrica.} {\bf 53}, (1985), 385--407.
%
\bibitem{crauel08}
H. Crauel.
\newblock{Measure attractors and Markov attractors.}
\newblock{\em Stochastics and Dynamics}, {\bf 23(1)} (2008) 75-107. 
%
\bibitem{crauel99}
H. Crauel, P. Imkeller, M. Steinkamp.
\newblock{Bifurcations of One-Dimensional Stochastic Differential Equations}
\newblock{\em Stochastics and Dynamics}, {} (1999) 27-47. 
%
\bibitem{csiszar}
I. Csisz\'ar, J. K\"orner.
\newblock Information theory: Coding theorem for discrete memoryless systems.
\newblock Cambridge University Press, 2011.
%
\bibitem{duffie}
D. Duffie, D. Filipovic, W. Schachermayer.
\newblock{Affine processes and applications in finance.}
\newblock{\em Ann. Appl. Probab.}, {\bf 13}, no. 3, (2003), 984--1053.
%
\bibitem{prato}
G. Da Prato, J. Zabczyk.
\newblock{\em Ergodicity for infinite dimensional systems.}
\newblock{ Vol. 229, Cambridge University Press}, 1996.
%
\bibitem{deelstra1}
G. Deelstra, F. Delbaen.
\newblock{Long term returns in stochastic interest rate models.}
\newblock{\em Insurance: Mathematics and Economics.} {\bf 17}, (1995), 163--169.
%
\bibitem{deelstra2}
G. Deelstra, F. Delbaen.
\newblock{Long term returns in stochastic interest rate models: Convergence in law.}
\newblock{\em Stochastics and Stochastics Reports.} {\bf 55}, (1995), 253--277.
%
\bibitem{deelstra3}
G. Deelstra.
\newblock{Long term returns in stochastic interest models: applications.}
\newblock{\em Astin Bulletin}, Vol. {\bf 30}, No. {\bf 1}, (2000), 123--140.
%
\bibitem{delbaen}
G. Deelstra, F. Delbaen.
\newblock{Existence of solutons of stochastic differential equations related to the Bessel process.}
\newblock{\em preprint. Dept. Mathematics, ETH, Z\"urich, Switzerland, 1994.}
%
\bibitem{ewens}
W. J. Ewens.
\newblock{Mathematical population genetics. {I}}
\newblock{\em Springer-Verlag, New York, 2004.}
%
\bibitem{feller51}
W. Feller.
\newblock{Two singular diffusion problems.}
\newblock{\em Annals of Mathematics.} {\bf 54}, (1951), 173--182.
%
\bibitem{feller54}
W. Feller.
\newblock{Diffusion processes in one dimension.}
\newblock{\em Trans. Amer. Math. Soc.} {\bf 77}, (1954), 1--31.
%
\bibitem{fu1}
M. Fukushima.
\newblock{Dirichlet forms and markov process.}
\newblock{\em North Holland Publishing, Amsterdam} {\bf 9}, (1980), 107--123.
%
\bibitem{fukushima}
M. Fukushima, D. Stroock.
\newblock{Reversibility of solutions to martingale problems.}
\newblock{\em Adv. Math. Suppl. Stud.} {\bf 9}, (1986), 107--123.
%
\bibitem{gross}
L. Gross.
\newblock{Logarithmic Sobolev inequalities.}
\newblock{\em Amer. J. Math.} {\bf 97}, No. {\bf 4}, (1975), 1061--1083.
%
\bibitem{hairer}
M. Hairer.
\newblock Convergence of Markov processes.
\newblock Lecture notes, (2016). 
% 
\bibitem{harris}
T. Harris.
\newblock The existence of stationary measures for certain Markov pr
ocesses. 
\newblock {\em Proceedings of the Third Berkeley Symposium on Mathematica
l Statistics and Probability}, (1954–1955), Vol. {\bf II}, 113--124. University of California Press, Berkeley and LosAngeles, 1956.
%
\bibitem{ikeda}
N. Ikeda, S. Watanabe.
\newblock{\em Stochastic differential equations and diffusion processes.} 
\newblock{Amsterdam-Oxford-New York}, 1981. 
%
\bibitem{jordan}
R. Jordan, D. Kinderlehrer, F. Otto.
\newblock{Free energy and the Fokker-Planck equation.}
\newblock{\em Physica D} {\bf 2-4}, (1997), 265--271. Landscape paradigms in physics and biology (Los Alamos, NM, 1996). 
%
\bibitem{neuenkirch}
A. Neuenkirch, L. Szpruch.
\newblock{First order strong approximations of scalar SDEs defined in a domain.}
\newblock{\em Numerische Mathematik}, Vol. {\bf 128}, Iss. {\bf 1}, (2014), 103--136.
%
\bibitem{schmalfuss}
B. Schmalfu{\ss}.
\newblock{Measure attractors and random attractors for stochastic
 partial differential equations.}
\newblock{\em Stochastic Anal. Appl.}, {\bf 17(6)}, (1999), 1075--1101. 
%
\bibitem{szpruch}
L. Szpruch, X. Mao, D.J. Higham, J. Pan.
\newblock{Numerical simulation of a strongly nonlinear Ait-Sahalia type interest rate model.}
\newblock{\em BIT,} {\bf 51(2)}, (2011), 405--425.
%
\bibitem{dat}
T. D. Tran, J. Hofrichter, J. Jost.
\newblock{The free energy method for the Fokker-Planck equation of the Wright-Fisher model.}
\newblock{\em submitted}
%
\bibitem{yamada1}
T. Yamada, S. Watanabe.
\newblock{On the uniqueness of stochastic differential equations.}
\newblock{\em J. Math. Kyoto Uni.}, {\bf 11}, (1971), 155--167.
%
\bibitem{yamada2}
T. Yamada, S. Watanabe.
\newblock{On the uniqueness of stochastic differential equations II.}
\newblock{\em J. Math. Kyoto Uni.}, {\bf 11}, (1971), 553--563.







\end{thebibliography}
\end{document}